\documentclass[9pt]{amsart}

\usepackage{amssymb}
\usepackage{amsmath}
\usepackage{amsthm, times, fullpage, float}
\usepackage{caption}
\usepackage{pstricks}
\usepackage{color}

\usepackage{palatino}

\newtheorem{theo}{Theorem}[section]

\theoremstyle{definition}

\newtheorem{definition}[theo]{Definition}

\theoremstyle{plain}
\newtheorem{lemma}[theo]{Lemma}
\newtheorem{theorem}[theo]{Theorem}

\newtheorem{proposition}[theo]{Proposition}

\theoremstyle{definition}

\newtheorem{remark}[theo]{Remark}

\newcommand{\beq}{\begin{equation}}
\newcommand{\eeq}{\end{equation}}
\renewcommand{\a}{\alpha}
\renewcommand{\b}{\beta}

\renewcommand{\d}{\delta}

\newcommand{\g}{\gamma}

\renewcommand{\o}{\omega}

\newcommand{\ga}{\mathfrak{a}}

\renewcommand{\gg}{\mathfrak{g}}
\newcommand{\gh}{\mathfrak{h}}
\newcommand{\gk}{\mathfrak{k}}
\newcommand{\gl}{\mathfrak{l}}
\newcommand{\gm}{\mathfrak{m}}
\newcommand{\gn}{\mathfrak{n}}

\newcommand{\gq}{\mathfrak{q}}

\newcommand{\gs}{\mathfrak{s}}
\newcommand{\gt}{\mathfrak{t}}
\newcommand{\gu}{\mathfrak{u}}

\newcommand{\gz}{\mathfrak{z}}

\newcommand\SU{\mathrm{SU}}

\renewcommand{\square}{\kern1pt\vbox
{\hrule height 0.6pt\hbox{\vrule width 0.6pt\hskip 3pt
\vbox{\vskip 6pt}\hskip 3pt\vrule width 0.6pt}\hrule height0.6pt}\kern1pt}

\DeclareMathOperator\Ad{Ad}
\DeclareMathOperator\ad{ad}

\newcommand{\n}{\nabla}

\newcommand{\be}{\begin{equation}}
\newcommand{\ee}{\end{equation}}

\def\<#1,#2>{\langle\,#1,\,#2\,\rangle}
\newcommand{\arr}{\begin{array}{rlll}}
\newcommand{\ea}{\end{array}}
\newcommand{\bea}{\begin{eqnarray}}
\newcommand{\eea}{\end{eqnarray}}
\newcommand{\bean}{\begin{eqnarray*}}
\newcommand{\eean}{\end{eqnarray*}}

\def\sideremark#1{\ifvmode\leavevmode\fi\vadjust{
\vbox to0pt{\hbox to 0pt{\hskip\hsize\hskip1em
\vbox{\hsize3cm\tiny\raggedright\pretolerance10000
\noindent #1\hfill}\hss}\vbox to8pt{\vfil}\vss}}}
\newcounter{ssig}
\newcounter{ttig}
\title[Homogeneous Hermitian manifolds and special metrics] {Homogeneous Hermitian manifolds and special metrics}

\author{Fabio Podest\`a}
\address{Dipartimento di Matematica e Informatica "Ulisse Dini", Universit\`a di Firenze, V.le Morgagni 67/A, 50100 Firenze, Italy}
\email{podesta@unifi.it}


\begin{document}
\begin{abstract} We consider non-K\"ahler compact complex manifolds which are homogeneous under the action of a compact Lie group of biholomorphisms and we investigate the existence of special (invariant) Hermitian metrics on these spaces. We focus on a particular class of such manifolds comprising the case of Calabi-Eckmann manifolds and we prove the existence of an invariant Hermitian metric which is Chern-Einstein, namely whose second Ricci tensor of the associated Chern connection is a positive multiple of the metric itself. The uniqueness is also discussed.  \end{abstract}

\maketitle
\section{Introduction}
A generalized flag manifold, namely a simply connected homogeneous space $M=G/K$ where $G$ is a compact Lie group and $K$ is the centralizer of a torus in $G$, can be endowed with invariant complex structures and invariant K\"ahler metrics. Once we fix an invariant complex structure $J$ on $M$, it is well known that there exists precisely one invariant K\"ahler-Einstein metric which is somehow canonically associated to $(M,J)$. A simply connected complex homogeneous space $G/H$ where $G$ is compact and $H$ is a compact subgroup is K\"ahler precisely when $H$ coincides with the centralizer of a torus in $G$, while the problem of finding special (invariant) Hermitian metrics on non-K\"ahler $G/H$ is far from being obvious. \par 
Given an Hermitian manifold $(M,J,g)$, there are several connections $D$ which leave both the metric $g$ and the complex structure $J$ parallel. Among them, the Chern connection is the only one with such a property and moreover the torsion $T$ being of type $(2,0)$. The curvature tensor $R$ of the Chern connection can be traced in two different ways yielding two different Ricci tensors $S^{(1)}$ and $S^{(2)}$ which are both Hermitian. While the $(1,1)$ form $\rho$ which can be associated to $S^{(1)}$ is closed and represents the first Chern class $c_1(M)$, the form associated to $S^{(2)}$ is not even closed and there are no obvious relations between these two tensors. When the second Chern-Ricci $S^{(2)}$ is positive definite (or at least nonnegative and positive at least at one point), then the Hodge numbers $h^{p,0}=0$ and therefore the arithmetic genus $\chi(M,\mathcal O)=1$ (see \cite{LY}). More recently (\cite{LY}) the second Chern-Ricci tensor has been involved in defining a Hermitian flow $\frac d{dt}h_t = -S^{(2)}(h_t)$ (called HCF in the sequel), which preserves the Hermitian structure, is strictly parabolic and coincides with the K\"ahler-Ricci flow whenever the initial metric is K\"ahler. This flow is actually a simplified version of the hermitian curvature flow introduced and studied by Streets and Tian (\cite{ST}) and in \cite{U} it has been recently proved that on a compact Hermitian manifold the HCF preserves the Griffiths non-negativeness of the Chern curvature. \par 
From this point of view, Hermitian metrics $h$ which are Chern-Einstein, namely whose second Ricci tensor $S^{(2)}$ satisfies $ S^{(2)} = \mu h$ for some $\mu\in C^\infty(M)$, are a distinguished class of metrics which has been first introduced in \cite{G1} and which deserves a special attention. It can be proved that on a generalized flag manifold there might exist several Chern-Einstein invariant metrics beyond the standard K\"ahler-Einstein metric, which on the other hand turns out to be the only one in some particular cases.\par 
In this paper we start the investigation of the existence of special metrics on compact simply connected complex homogeneous spaces and in particular we focus on a special subclass $\mathcal C$ of such homogeneous manifolds that includes the Calabi-Eckmann manifolds. A complex manifold in the class $\mathcal C$ is a ${\rm T}^2$-bundle over the product of two compact Hermitian symmetric spaces and can be endowed with a two-parameter family of inequivalent invariant complex structures. We prove that these manifolds, which are non-K\"ahler, do not satisfy the $\partial\bar\partial$-lemma, do not support any balanced nor SKT metrics, while for every invariant complex structure there exists an invariant Chern-Einstein metric with $\mu=1$. We also prove that this special metric is unique whenever the complex structure belongs to a suitable neighborhood of the so called standard complex structure on the manifold.  \par 
In Section 2 we recall some basic facts about the compact complex manifolds which are homogeneous under the action of a compact Lie group of biholomophisms. We discuss the $\partial\bar\partial$-lemma and we then focus on special homogeneous manifolds, called {\rm M}-manifolds, and a special subclass $\mathcal C$ 
which comprises the Calabi-Eckmann manifolds. We then discuss the existence of balanced metric on {\rm M}-manifolds or their suitable products.\par 
In Section 3 we review some basic notions about the Chern connection, we introduce the definition of Chern-Einstein metric and give some basic properties. We then state our main result, as Theorem \eqref{main}, where we state that a manifold in the class $\mathcal C$ carries an invariant Hermitian Chern-Einstein metric, but no balanced, nor SKT metric. We then describe the Chern connection of an invariant metric and its curvature algebraically, providing then a proof of our main result. We conclude with a remark on the behaviour of the HCF in a suitable neighborhood of a Chern-Einstein solution on a particular manifold in $\mathcal C$. \par 
In Section 4 we discuss the existence of invariant balanced metrics on compact complex homogeneous spaces.\par 
 \bigskip
\noindent{\bf Acknowledgment.} We like to thank Andrea Spiro, Daniele Angella and Valentino Tosatti for valuable conversations.

\section{Homogeneous complex manifolds}
Let $M$ be a compact complex manifold with complex structure $J$ and let $G$ be a
compact connected Lie group acting almost effectively, transitively and holomorphically on $(M,J)$. We
will write $M = G/L$ for some compact subgroup $L$. \par
The complexified group $G^c$ acts holomorphically on $G/L$, so that $M=G^c/U$ for
some complex subgroup $U \subset G^c$.
It is well known that the {\em Tits fibration} $\phi$ provides a
holomorphic fibering of the homogeneous space $M$ onto a compact rational homogeneous space $Q:=G^c/P$, where the
parabolic subgroup $P$ is in general defined to be the normalizer
$N_{G^c}(U^o)$ of $U^o$ (see \cite{Ak}). \par 
We will now suppose that $G$ is {\it semisimple} and that $M$ is supposed to be simply connected. Then $U$ (and $L$) is connected and the fibres of $\phi$ are complex tori. The flag manifold $G^c/P$ can be written as
$G/H$ endowed with a $G$-invariant complex structure $I$,
where $H$ is the centralizer of some torus in $G$.
Accordingly the Lie algebra $\gg$ can be decomposed as
\begin{equation}\label{dec}
\gg=\gl\oplus\underbrace{\gt\oplus\gn}_{\gm}\,,\end{equation}
where $\gm$ identifies with the tangent space $T_{[eL]}M$, $\gt$ with the tangent space of the fiber $T_{[eL]}F$, $F=\phi^{-1}([eH])$, $\gh=\gl\oplus\gt$ and $\gn$ is an $\Ad(H)$-invariant complement of
$\gh$ in $\gg$. The fiber $F$ being a complex torus implies 
$[\gt,\gt]=\{0\}$. Moreover the algebra $\gh$ is contained in the normalizer of $\gl$ in $\gg$
by construction, hence $[\gl,\gt]\subset\gl \cap \gt = \{0\}$ and $\gt$ is in the center of $\gh$.\par 
We can choose a Cartan subalgebra $\gs$ of the form $\gs=\gt_\gl^c\oplus
\gt^c$, where $\gt_\gl$ is a maximal abelian subalgebra of $\gl$.
Denote by $R$ the corresponding root system of $\gg^c$, by $R_\gl$ the subsystem relative to $\gl$ so that
$\gl^c = \gt_\gl^c \oplus \bigoplus_{\alpha\in R_\gl}\gg_\alpha$, and by $R_\gn$ the
symmetric subset of $R$ such that $\gn^c=\bigoplus_{\alpha \in R_\gn}\gg_\alpha$. The $G$-invariant complex structure
$I$ induces an endomorphism of $\gn^c$ that is $\Ad(H)$-invariant and therefore the corresponding subspace
$\gn^{1,0}$ is a sum of root spaces. The integrability of $I$ is equivalent to the condition
$$[\gn^{1,0},\gn^{1,0}]_{\gn^c} \subseteq \gn^{1,0}$$
and one can prove (see e.g. \cite{BFR}) that there is a suitable ordering of $R_\gn=R_\gn^+ \cup R_\gn^-$ such that
$$\gn^{1,0} = \bigoplus_{\alpha\in R_\gn^+} \gg_\alpha, \quad
\gn^{0,1} = \bigoplus_{\alpha\in R_\gn^-} \gg_\alpha.$$
The $G$-invariant complex structure $J$ on $G/L$ induces an $\Ad(L)$-invariant endomorphism, still denoted by $J$,
of $\gm^c$, where $\gm := \gt + \gn$. It leaves both $\gt$ and $\gn$ invariant with $J|_\gn =  I$ and
the integrability of $J$ is equivalent to the vanishing of the Nijenhuis tensor $N_J$, namely  for $X,Y\in \gm$
\begin{equation}\label{NI}
[JX,JY]_\gm - [X,Y]_\gm - J[JX,Y]_\gm - J[X,JY]_\gm = 0.
\end{equation}
Equation \eqref{NI} is trivial for $X,Y\in \gt$ and with $X\in \gt$ and
$Y\in \gn$ it reduces to the $\ad(\gt)$-invariance of $I$. When $X,Y\in \gn$, then \eqref{NI} is the integrability of $I$ 
because $[\gn^{1,0},\gn^{1,0}]\subseteq \gn^{1,0}$. \par
Viceversa, we start with a decomposition as in
\eqref{dec}, where $\gl+\gt = \gh$ and $\gh$ is the centralizer of an abelian
subalgebra.
If we fix an $\ad(\gh)$-invariant integrable complex
structure $I$ on $\gn$ and we extend it by choosing an arbitrary complex structure $J_\gt$ on $\gt$,
then $J_\gt + I$ will provide an integrable $L$-invariant complex structure $J$ on the homogeneous space $G/L$.\par 
Note that $(G/L,J)$ is K\"ahler if and only if $\gt=\{0\}$, i.e. $L=H$. We are mainly interested in the non-K\"ahler case.  
\medskip
\begin{proposition} The compact complex manifold $(G/L,J)$ does not satisfy the $\partial\bar\partial$-Lemma if it is not K\"ahler.\end{proposition}
\begin{proof} We fix a nonzero element $\xi\in \gt$ and consider the $\Ad(L)$-invariant element $\xi^*\in \gg^*$ given by the dual of $\xi$ w.r.t. the
$\Ad(G)$-invariant inner product $B$ on $\gg$. We consider the $2$-form $\omega=d\xi^*$, where we still denote by $\xi^*$ the $G$-invariant $1$-form 
on $M$ determined by $\xi^*$. We claim that $\omega$ is a non zero $(1,1)$-form and that it cannot be written as $\o = \partial\bar\partial f$ for 
$f\in C^\infty(M)$. The last assertion is clear, because $\o$ is $G$-invariant and the function $f$ can be chosen to be invariant as well, hence a constant. We show that $\o$ is not trivial. Indeed, if $\a\in R_\gn$ we have
$$\o(E_\a,E_{-\a}) = -B(\xi,[E_\a,E_{-\a}]) = B([E_\a,\xi],E_{-\a}) = -\a(\xi)B(E_\a,E_{-\a}).$$
We select a root $\a\in R_\gn$ so that $\a(\xi)\neq 0$ and our claim follows. In order to prove that $\o$ is of type $(1,1)$, we complexify it and 
observe that $\o(X,Y)\neq 0$ if and only if $X\in \gn^{1,0}$ and $Y\in \gn^{0,1}$ and therefore $\o(JX,JY) = \o(X,Y)$ holds. \end{proof}
While flag manifolds are K\"ahler and do have a special (invariant) metric which is represented by the unique K\"ahler-Einstein metric, for non-K\"ahler homogeneous spaces the question about the existence of special (invariant) Hermitian metrics is meaningful and deserves a special investigation.\par 
Following \cite{W}, an interesting class of complex homogeneous spaces is provided by {\rm M}-manifolds and their products. Given a compact simply connected Lie group $G$, a
{\rm M}-manifold is a $G$-homogeneous space of the form $G/L$ where $L$ is a subgroup of $G$ which coincides with the semisimple part of the centralizer of a torus in $G$. Using the fact that the semisimple $L$ has finite fundamental group we see that $G/L$ is simply connected and has finite second fundamental group. Moreover an even-dimensional {\rm M}-manifold and the product of two odd-dimensional ones carry infinitely many non-equivalent $G$-invariant complex structures (see \cite{W}). Simple examples of this situation is given by the Calabi-Eckmann manifolds, which can be described in group theoretic way as $\SU(n_1)\times\SU(n_2)/(\SU(n_1-1)\times \SU(n_2-1))\cong S^{2n_1-1}\times S^{2n_2-1}$.\par 
\begin{proposition}\label{bal0} The even-dimensional {\rm M}-manifolds or the product of two odd-dimensional {\rm M}-manifolds do not admit any balanced metric.\end{proposition}
\begin{proof} Indeed, let $M$ be such a manifold, which is the total space of a holomorphic toric fibration $\pi$ over a flag manifold $Q$. Since $M$
is simply connected and has finite second fundamental group, we see that $H_{2n-2}(M)$ ($n=\dim_{\mathbb C}M$) is trivial. Suppose $M$ admits a balanced metric whose K\"ahler form therefore satisfies $d(\o^{n-1})=0$. If $Z\subset Q$ is a codimension one compact submanifold of $Q$, then $\tilde Z = \pi^{-1}(Z)$ is a codimension one compact submanifold of $M$ that bounds and therefore $\int_{\tilde Z} \o^{n-1} = 0$, a contradiction. \end{proof}
\begin{remark} Note that the $G$-invariance of any balanced metric in the above proposition is not assumed - however, when a balanced metric exists, we can always find an invariant one (see \cite{FG}). Note also that this result has to be contrasted with the K\"ahler case. Indeed, {\it any} invariant Hermitian metric $h$ on a flag manifold $G/K$ {\it is} balanced, since the codifferential $\delta\o$ of the K\"ahler form $\o$ of $h$ is a $G$-invariant $1$-form and a flag manifold supports no non-trivial invariant $1$-forms.\par 
In Section 4 we will give a more detailed description of invariant balanced metrics. \end{remark}

We will now focus on a particular class $\mathcal C$ of homogeneous non-K\"ahler complex spaces given by a product of two {\rm M}-manifolds. This class $\mathcal C$ comprises the Calabi-Eckmann manifolds. We first describe them as a homogeneous space and then we study special hermitian invariant metrics on them. \par 
We consider two irreducible compact Hermitian symmetric spaces $G_1/T^1\cdot H_1$, $G_2/T^1\cdot H_2$, where $G_1,G_2$ are two compact simply connected simple Lie groups. The product of the corresponding {\rm M}-manifolds provides a homogeneous manifold 
$$M:= (G_1/H_1) \times (G_2/H_2)$$
which can be endowed with a family of invariant complex structures, already considered in \cite{T} (see also the more recent results in \cite{ST1} concerning also non invariant complex structures). Indeed, we consider the Cartan decompositions 
$$\gg_i = \underbrace{\mathbb R\cdot Z_i \oplus \gh^s_i }_{\gh_i} \oplus \gn_i,\qquad [\gn_i,\gn_i]\subseteq \gh_i,\qquad i=1,2,$$
where  $\gh^s_i$ denotes the simple part of $\gh_i$ and $Z_i$ in the center of $\gh_i$ determines the complex structure $I_i$ on $\gn_i$ by $I_i = \ad(Z_i)$. Therefore we have $\gl := \gh^s_1\oplus \gh^s_2$ and $\gt$ is spanned by $Z_1,Z_2$. The complex structure $J_\gt\in \mbox{End}(\gt)$ can be represented by the matrix $\left(\begin{smallmatrix} a& {\frac{-1-a^2}b}\\ b&-a\end{smallmatrix}\right)$ w.r.t. the basis $\{Z_1,Z_2\}$, where $a,b\in\mathbb R$, $b\neq 0$. The complex structure $J_o$ with $a=0,b=1$ will be called {\it standard}. \par 

\section{ The Chern connection and the main Theorem}
\par \medskip
Given a Hermitian manifold ($M,h,J$), the associated Chern connection is the unique hermitian connection, i.e. which leaves $h$ and $J$ parallel, and such that its torsion tensor $T$ is of type $(2,0)$, namely 
\beq\label{tor}T(JX,Y) = JT(X,Y)\eeq
for every vector fields $X,Y$ on $M$. Since the torsion of any hermitian connection on a complex manifold has vanishing $(0,2)$-component, \eqref{tor} is equivalent to (see \cite{G}) 
\beq \label{tor1} T(JX,JY)=-T(X,Y)\eeq
for every $X,Y$. This in turn is equivalent to saying that $T(Z,\overline W)=0$ for sections $Z,W\in \Gamma(T^{10}M)$. The curvature $R$ is a section of 
$\Lambda^{1,1}(T^*M)\otimes\gu(TM)$ and in local holomorphic coordinates it has the expression 
$$R_{i\bar j k \bar l } = -\frac{\partial h_{k\bar l}}{\partial z_i\partial \overline{z}_j} + h^{p\bar q}\frac{\partial h_{k\bar q}}{\partial z_i}
\frac{\partial h_{p\bar l}}{\partial \bar z_j}.$$
The curvature $R$ can be traced in two different ways. The {\it first} Ricci tensor $S^{(1)}$ is defined by tracing the endomorphism part, namely 
$$S^{(1)}_{i\bar j} = h^{p\bar q}R_{i\bar jp\bar q} = -\frac{\partial^2 \log(\det(h))}{\partial z_i\partial \overline{z}_j} $$
and its associated $(1,1)$ form is closed and represents the first Chern class $c_1(M)$. The second Ricci tensor is given by the trace 
$$S^{(2)}_{i\bar j} = h^{p\bar q}R_{p\bar qi\bar j}$$
and still there exists an associated $(1,1)$ form, which is not necessarily closed. The two Ricci tensors differ by a term which depends on the covariant derivative of the torsion (see e.g. Lemma 2.4 in \cite{ST}). It is known (see \cite{LY},\cite{KW})  that when $S^{(2)}$ is positive definite (negative definite resp.) the Hodge numbers $h^{p,0}=0$ for $p=1,\ldots,\dim_{\mathbb C}M$ ($M$ has no holomorphic vector fields resp.). We are led to the following definition, which was first considered in \cite{G1}
\begin{definition} A Hermitian metric $h$ is called Chern-Einstein if there exists $\mu \in C^\infty(M)$ so that 
$$S^{(2)} = \mu\cdot  h.$$
\end{definition}
\begin{remark} It follows using general formulas (see e.g. \cite{G2}, ~p. 501) that a metric conformal to a K\"ahler-Einstein metric turns out to be Chern-Einstein. For this reason the original definition in \cite{G1} included the hypothesis that the metric is Gauduchon, namely $\partial\bar\partial(\o^{n-1}) =0$, in order to exclude this less significant situation. In case of homogeneous manifold, it is clear that $\mu$ has to be constant and indeed every invariant Hermitian metric is Gauduchon. In \cite{LY} and \cite{G1}  it is shown that the canonical metric on the Hopf manifold $S^{2n+1}\times S^1$ is Chern-Einstein with $\mu >0$. This also shows that, in contrast with the K\"ahler-Einstein case, the existence of a positive Chern-Einstein metric does not imply the simply connectedness of the manifold. Nevertheless, a compact complex manifold $M$ with finite fundamental group is simply connected when it carries a positive Chern-Einstein metric. This indeed follows from the the fact the aritmetic genus  $\chi(M,\mathcal O) =1$ and this invariant is actually multiplicative with finite coverings.\end{remark} 
\begin{remark} We also recall that there exists a third Ricci tensor $S^{(3)}$ which is defined as $S^{(3)}_{\a,\bar\b} := h^{p\bar q}R_{\a\bar q p\bar\b}$. The Einstein condition $S^{(1)} = \mu h$ or $S^{(3)} = \mu h$ for some constant $\mu\neq 0$ is easily seen to imply $h$ to be K\"ahler (see \cite{B}).
\end{remark}
We may now state our main result
\begin{theorem}\label{main} Let $M$ be a manifold in the class $\mathcal C$ endowed with an invariant complex structure $J$. Then $M$ is simply connected, non K\"ahler and 
\begin{itemize}
\item[a)] $M$ does not admit any balanced or SKT Hermitian metric;
\item [b)] $M$ admits an invariant Hermitian metric $\bar g$ which is Chern-Einstein with $S^{(2)}(\bar g) = \bar g$. Moreover, if $J$ belongs to a suitable neighborhood of the standard complex structure , the metric $\bar g$ is the only invariant Chern-Einstein metric satisfying $S^{(2)}(\bar g) = \bar g$;
\item[c)] $c_1(M)\geq 0$.
\end{itemize}
\end{theorem}
Before starting with the proof of the main Theorem, we describe the Chern connection of an invariant Hermitian metric and prove some basic facts.\par

Given an invariant Hermitian metric $h$ on a complex homogeneous space $G/L$, we will describe its associated Chern connection $\nabla$. \par 
We see $h$ as an $\Ad(L)$-invariant inner product $h$ on the $\Ad(L)$-invariant complement $\gm$ with $\gg = \gl \oplus \gm$. Moreover $h$ is supposed to be Hermitian w.r.t. the invariant complex structure $J$ on $\gm$. Being $G$-invariant, the torsion $T$ can be seen as an element of $\Lambda^2\gm\otimes\gm$ and after complexification, the condition \eqref{tor1} is equivalent to
\beq\label{tor3} T(\gm^{10},\gm^{01})=0.\eeq   
Since $\nabla$ is an invariant connection on $G/L$, it is well known that it is completely determined by a map $\Lambda\in \mbox{Hom}(\gm,\mbox{End}(\gm))$, where the 
correspondence can be described as follows. If $X\in \gg$ we denote by $X^*$ the corresponding vector field on $M=G/L$ and observe that the map 
$\gm\ni X\mapsto (X^*)|_{[eL]} \in T_{[eL]}M$ is an isomorphism. Then for $X,Y\in \gm$ and $p=[eL]\in M$,
$$(\Lambda(X)Y)^*|_p = (\nabla_{X^*}Y^* - [X^*,Y^*])|_p.$$
\begin{lemma} The condition $\nabla J = 0$ implies $[\Lambda(X),J]=0$.
\end{lemma}
\begin{proof} Using the definition of $J$ as an endomorphism of $\gm$ (namely $(JX)^*|_p = J X^*|_p$) we see that 
$$(\n_XJY)^*|_p = \n_{X^*}(JY)^* = (\n_{(JY)^*}X^* + [X^*,(JY)^*] + T(X^*,(JY)^*))|_p = $$
$$= (\n_{JY^*}X^* + [X^*,(JY)^*] + T(X^*,JY^*))|_p =$$
$$= (J\n_{X^*}Y^* + [J Y^*,X^*] + T(J Y^*,X^*) + [X^*,(JY)^*] + T(X^*,JY^*))|_p =$$
$$= (J\n_{X^*}Y^* - J[X^*,Y^*])|_p  + [X^*,(JY)^*]|_p = J(\Lambda(X)Y)^*|_p + [X^*,(JY)^*]|_p$$
and our claim follows.\end{proof}
Therefore when we extend $\Lambda$ to an element $\Lambda\in \mbox{Hom}(\gm^c,\mbox{End}(\gm^c))$, we have that $\Lambda(\cdot)(\gm^{10})\subseteq\gm^{10}$.
The torsion is given by (see \cite{KN},~p.192) 
$$T(X,Y) = \Lambda(X)Y - \Lambda(Y)X - [X,Y]_\gm$$
and condition \eqref{tor3} implies for $A,B\in \gm^{10}$ 
$$0 = (\Lambda(A)\bar B - \Lambda(\bar B)A - [A,\bar B])^{10} = $$
$$ = (- \Lambda(\bar B)A - [A,\bar B])^{10},$$
i.e.
\beq \Lambda(\bar B)A = (\Lambda(\bar B)A)^{10} =  [\bar B,A]^{10}.\eeq
Conjugation yields 
\beq\label{AbarB} \Lambda(A)\bar B = [A,\bar B]^{01}.\eeq
\subsection{The proof of the main Theorem}
\medskip
We already know that any manifold $M$ in $\mathcal C$ does not admit any balanced metric. We start here with some generalities in order to prove our main result.\par   
Using the same notations as above, we consider the Cartan subalgebra $\ga_i$ in $\gh_i^c$ given by $\gt_i^c\oplus \gs_i$, where $\gt_i = \mathbb R\cdot Z_i$ and $\gs_i$ is a Cartan subalgebra of $(\gh^s_i)^c$ for $i=1,2$. We denote by $R^{(i)}$ the corresponding root systems, which are then endowed with an invariant ordering corresponding to the invariant complex structure $I_i$ on $\gn_i$ ($i=1,2$). Then $R^{(i)}= R_{\gh_i}\cup R_{\gn_i}^+ \cup R_{\gn_i}^-$. In the sequel we will extend each root $\a\in R^{(i)}$ to a functional on $\ga_1\oplus \ga_2$ by putting $\a|_{\ga_j}\equiv 0$ if $i\neq j$.\par 
We are interested in studying special Hermitian metrics on these manifolds. We recall that a Hermitian invariant metric $h$ on $M$ is given by an Hermitian 
$\Ad(L)$-invariant inner product on $\gm = \gt +\gn$, where $\gn = \gn_1\oplus \gn_2$. Moreover, whenever $\dim\gh_i > 1$, $i=1,2$, the tangent space 
$\gm$ splits as the sum of three inequivalent, and therefore $h$-orthogonal, submodules $\gt,\gn_1,\gn_2$. We put $n_i:= \dim_\mathbb C\gn_i$ for $i=1,2$.\par 
Moreover, by $L$-irreducibility, the metric $h$ on each $\gn_i$ is a negative multiple  of the restriction of the Cartan Killing form $B_i$ of $\gg_i$ on $\gn_i$ ($i=1,2$). We will also denote by $H$ a non zero element of $\gt^{10}$, say $H= Z_1 - iJ_{\gt}Z_1$, and we  put 
$$h_o := h(H,\bar H)>0.$$
Note that if $\alpha,\b$ are  positive roots in $R_{\gn_i}^+$ then by the $\Ad(L)$-invariance for every $v\in \gs_i$ 
$$0=h([v,E_\a],E_{-\b}) + h(E_\a,[v,E_{-\b}]) = (\a-\b)(v)\cdot h(E_\a,E_{-\b}).$$
Since $\a\neq\b$ implies that $\a\not\equiv\b$ on $\gs_i$, we see that $h(E_\a,E_{-\b}) \neq 0$ only when $\a=\b$. In this case $h(E_\a,\bar E_\a) := g_i$ when we use the normalized root vectors $\{E_\a\}$ given by a Chevalley basis (recall also that $\bar E_\a = - E_{-\a}$). \par\medskip
\begin{lemma} $M$ does not admit any SKT metric.\end{lemma}
\begin{proof} Suppose $h$ is a SKT metric, which can be supposed to be invariant using the compactness of the group $G$. The SKT condition amounts to say that $dd^c \o = 0$, where $\o$ is the corresponding K\"ahler form. We will use the Koszul formula for the differential of an invariant $q$-form $\phi$, where for $v_0,\ldots,v_q\in \gm$ 
$$d\phi(v_o,\ldots,v_q) = \sum_{i<j}(-1)^{i+j}\phi([v_i,v_j]_\gm,v_0,\ldots,\tilde v_i,\ldots,\tilde v_j,\ldots, v_q),$$
where $\tilde{\cdot} $ indicates that the corresponding vector does not appear. Select $\a,\b\in R_{\gn_1}$ and we compute 
$$d(d^c\o)(E_\a,E_{-\a},E_\b,E_{-\b}) = - d^c\o((H_\a)_\gm,E_\b,E_{-\b}) - d^c\o((H_\b)_\gm, E_\a,E_{-\a}) = $$
$$= - d\o(J(H_\a)_\gm,E_\b,E_{-\b}) - d\o(J(H_\b)_\gm, E_\a,E_{-\a}).$$
Now
$$d\o(J(H_\a)_\gm,E_\b,E_{-\b}) = - \o(\b(J(H_\a)_\gm)E_\b,E_{-\b}) - \o(\b(J(H_\a)_\gm)E_{-\b},E_\b) - \o((H_\b)_\gm,J(H_\a)_\gm) = - h((H_\a)_\gm,(H_\b)_\gm)$$
is symmetric in $\a,\b$ and therefore we get 
$$0=d(d^c\o)(E_\a,E_{-\a},E_\b,E_{-\b}) = 2 g((H_\a)_\gm,(H_\b)_\gm).$$
Now $(H_\a)_\gm = \frac{B_1(H_\a,Z_1)}{B_1(Z_1,Z_1)} Z_1 = -\frac{\sqrt{-1}}{2n_1}Z_1$, so that last equation implies $h(Z_1,Z_1) =0$ , a contradiction. 
\end{proof}

In order to prove the existence of a Chern-Einstein metric, we prove the following

\begin{lemma}\label{Lambda} Given $\a,\b\in R_{\gn_i}^+$, $\g\in R_{\gn_j}^+$ with $i\neq j$, we have
\begin{itemize}
\item[a)] $\Lambda(E_\a)E_\b = 0$,\quad $\Lambda(E_\a)E_{\pm\g} = 0$ and $\Lambda(E_\a)\bar H=0 $;
\item[b)] $\Lambda(E_\a)E_{-\b} = 0$ for $\b\neq \a$;\quad $\Lambda(E_\a)E_{-\a} = -\frac {\sqrt{-1}}{2n_i}(Z_i)^{01}$; 
\item[c)] $\Lambda(E_\a)H= \frac{h(H,H_\a)}{g_i}E_\a$;
\item[d)] $\Lambda(v) = \ad(v)$ for $v\in\gt^c$.

\end{itemize}

\end{lemma}
\begin{proof} a) Given $A\in \gn^{10}$ we have by \eqref{AbarB} and by $[\gn_i,\gn_i]\subseteq \gh_i$ 
$$h(\Lambda(E_\a)E_\b,\bar A) = -h(E_\b,\Lambda(E_\a)\bar A) = -h(E_\b,[E_\a,\bar A]) = 0.$$
Moreover, 
$$h(\Lambda(E_\a)E_\b,\bar H) = -h(E_\b,\Lambda(E_\a)\bar H) = -h(E_\b,[E_\a,\bar H]) = h(E_\b,\a(\bar H)E_\a)=0.$$
The same kind of arguments shows the second assertion. Finally by \eqref{AbarB}, $\Lambda(E_\a)\bar H = [E_\a,\bar H]^{01} = \a(\bar H)E_\a^{01}=0$.\par 
\noindent b) We have $\Lambda(E_\a)E_{-\b} = [E_\a,E_{-\b}]^{01} \neq 0$ only if $\a=\b$ and in this case $\Lambda(E_\a)E_{-\a}=
[E_\a,E_{-\a}]^{01}= [H_\a]^{01}$. Now 
$$ [H_\a]^{01} = \frac{B_i(H_\a,Z_i)}{B_i(Z_i,Z_i)} Z_i^{01} = -\frac {\sqrt{-1}}{2n_i}Z_i^{01}.$$\par 
\noindent c) We have $h(\Lambda(E_\a)H,E_{-\a}) = - h(H,H_\a^{01}) = - h(H,H_\a)$ and our claim follows.\par 
\noindent d) First note that $\Lambda(H)\bar H = [H,\bar H]^{01} = 0$ and $\Lambda(H)E_{-\a}=-\a(H)E_{-\a}$ by \eqref{AbarB}. From this we see that $\Lambda(H)H=0$. Now 
$$h(\Lambda(H)E_\a,\bar H) = -h(\Lambda(H)\bar H,E_\a)=0,$$
$$h(\Lambda(H)E_\a,E_{-\b}) = -h(E_\a,-\b(H)E_{-\b}) = \d_{\a,\b}\a(H) h(E_\a,E_{-\a})$$
so that $\Lambda(H)E_\a= \a(H)E_\a$. From this we see that $\Lambda(\bar H)E_\a = -\overline{\a(H)}E_\a$. Since $\a|_{\gt}\in \sqrt{-1}\mathbb R$, we see that $\overline{\a(v)} = -\a(\bar v)$ for $v\in \gt^c$, hence $\Lambda(\bar H)E_\a= \a(\bar H)E_\a$. Similarly for $\Lambda(\bar H)E_{-\a} = -\a(\bar H)E_{-\a}$.\end{proof}
In order to compute the second Ricci tensor $S^{(2)}$ ($S$ for brevity throughout the following), we compute the curvature. We use the general formula for the curvature of an invariant metric (see e.g. \cite{KN}, p. 192) for $v,w\in \gm$
\beq\label{hcurv}R(v,w) = [\Lambda(v),\Lambda(w)] - \Lambda([v,w]_{\gm}) - \ad([v,w]_{\gl}).\eeq
Given $\a,\b\in R_{\gn_i}^+$ we have 
\beq\label{curv}R(E_{\a},\bar E_{\a})E_\b = [\Lambda(E_\a),\Lambda(\bar E_\a)]E_\b - \Lambda((H_\a)_{\gm^c})E_\b - [(H_\a)_{\gl^c},E_\b]=\eeq
$$= \Lambda(E_\a)\Lambda(\bar E_\a)E_\b - \Lambda((H_\a)_{\gm^c})E_\b - [(H_\a)_{\gl^c},E_\b]=$$
$$= \Lambda(E_\a)\Lambda(\bar E_\a)E_\b - \b(H_\a)E_\b,$$
where we have used Lemma \ref{Lambda}, (a)(d). Note also that $R(H,\bar H)E_\a=0$. \par 

Using Lemma \ref{Lambda} we see that for $\b\in R_{\gn_i}$
$$S(E_\b,\bar E_\b) = \sum_{\a\in R_{\gn_i}^+}\frac 1{g_i}h(R(E_\a,\bar E_\a)E_\b,\bar E_\b) + \frac 1{h_o}h(R(H,\bar H)E_\b,\bar E_\b) = $$
$$= -\sum_{\a\in R_{\gn_i}^+}\frac 1{g_i}h(\Lambda(\bar E_\a)E_\b,\Lambda( E_\a)\bar E_\b) -\sum_{\a\in R_{\gn_i}^+} \b(H_\a)= $$
$$ = -\frac 1{g_i}h(\Lambda(\bar E_\b)E_\b,\Lambda( E_\b)\bar E_\b) -\sum_{\a\in R_{\gn_i}^+} \b(H_\a)=$$
$$ = -\frac 1{g_i}h(H_\b^{01},\overline{H_\b^{01}}) -\sum_{\a\in R_{\gn_i}^+} \b(H_\a)=$$
$$ = \frac 1{g_i}h(H_\b^{01},H_\b^{10}) -\sum_{\a\in R_{\gn_i}^+} \b(H_\a)=$$
$$ = -\frac{1}{4g_in_i^2} h(Z_i^{10},Z_i^{01}) -\sum_{\a\in R_{\gn_i}^+} \b(H_\a)= -\frac 1{8g_in_i^2}h(Z_i,Z_i) + \frac 12,$$

where we have used that $\sum _{\a\in R_{\gn_i}^+} H_\a = 
-\frac {\sqrt{-1}}{2}Z_i$ and $\b(Z_i)=\sqrt{-1}$. Now it is immediate to see that 
$$h(Z_1,Z_1)=\frac 12 h_o,\qquad h(Z_2,Z_2)= \frac{1+a^2}{b^2}h(Z_1,Z_1) = \frac{1+a^2}{2b^2}h_o,$$
so that
\beq\label{S1} S(E_\b,\bar E_\b) =  -\frac {h_o}{16g_1n_1^2} + \frac 12,\qquad \b\in R_{\gn_1},\eeq
\beq\label{S} S(E_\b,\bar E_\b) =  -\frac {(1+a^2)h_o}{16b^2g_2n_2^2} + \frac 12,\qquad \b\in R_{\gn_2},\eeq
\par 
We now compute 
$$S(H,\bar H) = \sum_{i=1}^2\sum_{\a\in R_{\gn_i}^+}\frac 1{g_i}h(R(E_\a,\bar E_\a)H,\bar H) =$$
$$= -\sum_{i=1}^2 \sum_{\a\in R_{\gn_i}^+}\frac 1{g_i}h(\Lambda(\bar E_\a)\Lambda(E_\a)H,\bar H) = 
\sum_{i=1}^2 \sum_{\a\in R_{\gn_i}^+}\frac 1{g_i}h(\Lambda(E_\a)H,\Lambda(\bar E_\a)\bar H) =$$
$$= \sum_{i=1}^2 \sum_{\a\in R_{\gn_i}^+}\frac 1{g_i^2}|h(H,H_\a)|^2.$$
Now it is immediate to see that 
$$h(H,H_\a) = -\frac{\sqrt{-1}}{2n_1}h(Z_1,Z_1),\qquad \a\in R_{\gn_1}^+,$$
$$h(H,H_\a) = \frac{\sqrt{-1}}{2n_2}(-\frac ab - \frac{\sqrt{-1}}b)h(Z_1,Z_1),\qquad \a\in R_{\gn_2}^+$$
so that 
$$S(H,\bar H) = \frac{h_0^2}{16}\left( \frac{1}{g_1^2n_1} + \frac{1+a^2}{b^2g_2^2n_2}\right).$$

Summing up the Hermitian Einstein equations are 
\beq\label{eq} \left\{\begin{matrix}  \frac{h_0}{16}\left( \frac{1}{g_1^2n_1} + \frac{1+a^2}{b^2g_2^2n_2}\right) &= 1 \\
-\frac{h_o}{16n_1^2g_1} + \frac 12 &= g_1 \\ 
-\frac {(1+a^2)h_o}{16b^2g_2n_2^2} + \frac 12 &= g_2.    \end{matrix}\right.
\eeq 
Putting $x:= 1/g_1$, $y:= 1/g_2$ and $z:= 16/h_o$, the system \eqref{eq} can be written as 
\beq\label{eq} \left\{\begin{matrix}  z &=\frac 1{n_1}x^2 + \frac {1+a^2}{b^2n_2} y^2  \\
z(x-2)  &= \frac{2}{n_1^2} x^2 \\ 
z(y-2) &= \frac{2(1+a^2)}{b^2n_2^2} y^2     \end{matrix}\right.
\eeq 
which is equivalent to 
\beq\label{eq1} \left\{\begin{matrix}  z &=\frac 1{n_1}x^2 + \frac {1+a^2}{b^2n_2} y^2  \\
z(x-2)  &= \frac{2}{n_1^2} x^2 \\ 
n_1 x + n_2 y &= 2n_1+2n_2 + 2 .    \end{matrix}\right.
\eeq 
We have an admissible solution $x,y,z\in \mathbb R^+$ if and only if there is a solution $x$ of the polynomial equation 
\beq\label{eq}\phi(x):= \left[ \frac 1{n_1}x^2 + \frac {1+a^2}{b^2n_2}(\frac{2n_1+2n_2+2-n_1x}{n_2})^2\right]\cdot (x-2)- \frac 2{n_1^2}x^2 = 0\eeq
satisfying the conditions 
$$x> 0,\qquad x < \frac {2n_1+2n_2+2}{n_1}.$$
This follows immediately from the fact that 
$$\phi(0)<0,\qquad \phi(\frac {2n_1+2n_2+2}{n_1}) = \frac{8n_2(n_1+n_2+1)^2}{n_1^4} > 0.$$
We now put $\frac {1+a^2}{b^2}=1$ throughout the following and prove the uniqueness of the Hermitian Einstein metric. \par 
\begin{lemma} Any solution of the equation $\phi(x)=0$ satisfies $x\in [2,2+\frac 2{n_1}]$. If there are two distinct solutions $x_1<x_2$, then the equation 
$\phi'(x)=0$ has two distinct solutions $y_1<y_2$ in $[2,2+\frac 2{n_1}]$.
\end{lemma}
\begin{proof} It is immediate to see that $\phi(2)<0$ and $\phi(2+\frac 2{n_1}) = \frac 8{n_1n_2} > 0$. Moreover any solution $x$ satisfies 
$$0 < \frac 2{n_1^2}x^2 =  \left[ \frac 1{n_1}x^2 + (\frac{2n_1+2n_2+2-n_1x}{n_2})^2\right]\cdot (x-2) \geq \frac 1{n_1}x^2 \cdot (x-2), $$
hence $x\in [2,2+\frac 2{n_1}]$. The second claim follows immediately from the fact that $\phi'$ is a polynomial of degree $2$. \end{proof}
If we suppose that there are two distinct solutions $x_1,x_2$ of \eqref{eq}, then by the previous Lemma we get $2< \frac 12(y_1+y_2)< 2+\frac 2{n_1}$. Using Maple, this last condition is given by 
$$\left\{ \begin{matrix} 2 n_1 n_2^3 - 2 n_2n_1^3 - n_2^3 - 2 n_1^3 &\leq 0\\ 
2 n_2 n_1^3 - 2 n_1n_2^3 - n_1^3 - 2 n_2^3 &\leq 0\end{matrix}\right.$$
that can be rewritten as 
\beq\label{diseq}\left\{ \begin{matrix} \frac{n_1^3}{n_1+1} &\leq 2 \frac{n_2^3}{2n_2-1}\\ 
\frac{n_2^3}{n_2+1} &\leq 2 \frac{n_1^3}{2n_1-1}\end{matrix}\right.\eeq
Now we observe that if, say, $n_1 = 1$, then $n_2^3\leq 2n_2+2$, giving $n_2=1$ and similarly if $n_2=1$ we get $n_1=1$. The cubic equation $\phi(x)=0$ with $n_1=n_2=1$ can be easily checked to have only one solution, so that we can suppose $n_1,n_2\geq 2$. By \eqref{diseq} we see that 
\beq\label{diseq1}\left\{ \begin{matrix} n_1&\leq \sqrt{2} n_2\\ 
n_2&\leq \sqrt{2} n_1.\end{matrix}\right.\eeq
Now the discriminant $d$ of the equation $\phi'(x)=0$ is given by 
$$0< d:= n_1^6 + n_2^6 + 2 n_1^6 n_2 + 2 n_2^6 n_1 + n_1^6 n_2^2 + n_1^2 n_2^6 + n_1^3 n_2^3 + $$
$$- (8 n_1^4n_2^4 + 3 n_1^3n_2^5 + 3 n_1^5 n_2^3 + 2 n_1^3n_2^4 + 2 n_1^4 n_2^3)  $$
which is a symmetric expression in $n_1,n_2$. We can suppose $n_2\leq n_1$ and using $n_2\geq \frac1{\sqrt 2} n_1$ by \eqref{diseq1}, we see that 
$$0 < d \leq 2 n_1^8 + 4 n_1^7 + 3 n_1^6 - ((2+\frac 9{4\sqrt 2})n_1^8 + (\frac1{\sqrt 2}+\frac 12)n_1^7) =$$
$$= - n_1^6\left( \frac 9{4\sqrt 2} n_1^2 - (\frac 72 - \frac 1{\sqrt 2}) n_1 - 3\right) < 0$$
for $n_1\geq 3$. So we are left with the case $n_1 = n_2 = 2$. In general for $n_1=n_2$ the equation admits only one solution, which is explicitely given by 
$$g_1=g_2 = \frac {n_1}{2n_1+1},\ h_o = \frac {8 n_1^3}{(2n_1+1)^2}.$$
In order to prove (c) in Theorem \eqref{main}, we compute the first Chern-Ricci tensor $\rho$ of an invariant metric $h$. Using the general formula \eqref{hcurv} and Lemma \ref{Lambda},(d),  we see that for $\a\in R_{\gn_1}$,
$$\rho(E_\a,\bar E_\a) = -\sum_{\b\in R_{\gn_1}}\b(H_a) = \frac 12.$$
Similarly for $\a\in R_{\gn_2}$, we see that $\rho(E_\a,\bar E_\a) = \frac 12$. Since $R(H,\bar H)=0$, we have $\rho(H,\bar H)=0$ and therefore $\rho\geq 0$. \par\bigskip
\noindent{\bf The Chern-Ricci flow.} As a last remark, we consider the Chern-Ricci flow 
\beq\label{flow}h_t' = h_t - S(h_t),\eeq
which has a Chern-Einstein metric as an equilibrium point. It is known that there exists a solution for some interval $t\in [0,T)$ for any initial metric $\bar h$. Moreover it is immediate to observe that the solution $h_t$ still has the full group $G$ acting by isometric biholomorphisms. Using a special case 
given by some manifold $M\in \mathcal C$ with $n_1=n_2=2$, we see numerically that the long time existence is not guaranteed and that even when the initial 
 metric $\bar h$ has positive Ricci tensor $S$, the flow does not necessarily converge to the Chern-Einstein metric. 
\section{Invariant balanced metrics}
We keep the same notations as in the previous sections and we consider a complex homogeneous space $M=G/L$ of complex dimension $n$ as in Section 2. We like to study the existence of invariant balanced metrics. We recall that a Hermitian metric $h$ is called {\it balanced} if $d(\o^{n-1})=0$ where $\o$ denotes the K\"ahler form. This definition is actually equivalent to requiring that $\d\o = 0$, where $\d$ denotes the co-differential w.r.t. the metric $h$. \par 
We also recall that if a balanced metric exists, then an invariant balanced metric exists too (see \cite{FG}). We now focus on the possible construction of {\it adapted} balanced metrics on $G/L$,namely metrics which submerge an invariant Hermitian metric on the corresponding flag manifold $G/H$ with $\gt$ and $\gn$ being orthogonal (note that any invariant metric on $G/L$ is of this form whenever $\gh$ coincides with the centralizer in $\gg$ of its semisimple part). The condition of being adapted balanced has been already investigated in \cite{GGP}, Lemma 2; here we give a direct proof using some standard computations on the Levi Civita connection, which might be useful for further research. We start proving the following Lemma, where we denote by $D$ the Levi Civita connection of $h$.
\begin{lemma}\label{lemma1} The metric $h$ is balanced if and only if , given $\{e_i\}_{i=1,\ldots,2n}$ an orthonormal basis of the tangent space $\gm \cong T_{[eL]}M$
we have 
$$\sum_i JD_{e_i}e_i - D_{Je_i}e_i = 0.$$
\end{lemma}
\begin{proof} We know that $\d\o(v)=-\sum_i (D_{e_i}\o)(e_i,v)$ for $v\in \gm$. We extend any element of $\gm$ to the corresponding Killing vector field 
which will be denoted by the same letter with $^*$. We have 
$$-(D_{e_i^*}\o)(e_i^*,v^*) = - e_i^*\o(e_i^*,v^*) + \o(D_{e_i^*}e_i^*,v^*) + \o(e_i^*,D_{e_i^*}v^*) = $$
$$= - \o(e_i^*,[e_i^*,v^*]) + \o(D_{e_i^*}e_i^*,v^*) + \o(e_i^*,D_{e_i^*}v^*) = \o(D_{e_i^*}e_i^*,v^*) + \o(e_i^*,D_{v^*}{e_i^*}) = $$
$$= \o(D_{e_i^*}e_i^*,v^*) + h(Je_i^*,D_{v^*}{e_i^*}) = h(JD_{e_i^*}e_i^*,v^*) - h(v^*,D_{Je_i^*}{e_i^*})$$
and our claim follows.\end{proof}

We compute the Levi Civita connection using the standard formula (see e.g. \cite{KN}) for $v,w,z\in \gm$
\beq\label{LC} 2h(D_vw,z) = h([v,w]_\gm,z) + h([z,v]_\gm,w) + h([z,w]_\gm,v).\eeq
We immediately see that for every $v,w\in \gt$ we have $D_vw=0$ because $\ad(v)(\gn)\subseteq \gn$ and $h(\gt,\gn)=0$. For every $\a\in R_\gn^+$ we consider the vectors $e_\a:= \frac{E_\a-E_{-\a}}{\sqrt {2 g_\a}}$ so that $\{e_\a,Je_\a\}_{\a\in R_\gn^+}$ gives an orthonormal basis of $\gn$. A simple computation shows that 
\beq\label{D1} D_{e_\a}e_\a + D_{Je_\a}Je_\a = -\frac 1{g_\a}\left[ D_{E_\a}E_{-\a}+ D_{E_{-\a}}E_\a\right],\qquad  D_{Je_\a}e_\a - D_{e_\a}Je_\a = \frac{i}{g_\a}\left[D_{E_{-\a}}E_{\a} - D_{E_\a}E_{-\a} \right].\eeq
Now, using \eqref{LC} we see that for every root $\b\in R_\gn$, $v\in \gt^c$
$$h(D_{E_\a}E_{-\a},E_\b) = 0,\qquad h(D_{E_\a}E_{-\a},v) = \frac 12 h(H_\a,v),$$
so that 
$$D_{E_\a}E_{-\a} = \frac 12(H_\a)_{\gt^c},\qquad D_{E_{-\a}}E_{\a} = \overline{D_{E_{\a}}E_{-\a}} = - \frac 12 (H_\a)_{\gt^c}.$$
Therefore $\sum_{i=1}^{2n} D_{e_i}e_i = 0$ by \eqref{D1} and the condition in Lemma~\ref{lemma1} becomes $\sum_{i=1}^{2n} D_{Je_i}e_i = 0$. Therefore by \eqref{D1}, $h$ is balanced if and only if 
\beq\label{bal} \sum_{\a\in R_\gn^+} \frac 1{g_\a} H_\a|_{\gt^c} = 0\qquad\mbox{or\ equivalently}\quad \sum_{\a\in R_\gn^+} \frac 1{g_\a} H_\a \in \sqrt{-1}\gl\eeq
We define the vector 
\beq\label{def} \d_h:= \sum_{\a\in R_\gn^+} \frac 1{g_\a} H_\a\eeq
and note that $\d_h$ is a slight modification of the standard Koszul element $\d_{\kappa}:= \frac 12 \sum_{\a\in R_\gn^+} H_a$ which lies in $\sqrt{-1}\gz$, where $\gz$ is the center of $\gh$. Indeed, we can prove that $\d_h$ is a non zero vector in  $ \sqrt{-1}\gz$ by the following arguments. First of all we decompose $\gn^{1,0}=\oplus_{j=1}^s \gq_i$ as a sum of irreducible $\gh$-modules $\gq_j$, $j=1,\ldots,s$. Note that there exist $R_j\subset R_\gn^+$ with $\gq_j = \bigoplus_{\a\in R_j}\gg_\a$, $j=1,\ldots,s$. We prove the following 
\begin{lemma} Given  $\zeta_j:= \sum_{\a\in R_j}H_\a$, then $\sqrt{-1}\zeta_j\in\gz$.\end{lemma}
\begin{proof} We fix $\g\in R_\gk$ and for every $\b\in R_j$ we consider the maximal $\g$-string $\{\b+k\g,\ p\leq k\leq q\}$. Note that $(R_j\pm\g)\cap R\subseteq R_j$ by the $\Ad(H)$-invariance of $\gn_1$. This means that the whole $\g$-string belongs to $R_j$. Moreover 
$$\sum_p^q (\b+k\g,\g) = (q-p+1)(\b,\g) + \frac 12(q(q+1)+p(1-p)) ||\g||^2 = \frac {||\g||^2}2[-(p+q)(q-p+1) + q(q+1)+p(1-p)] = 0.$$
Since the whole $R_j$ splits up as the disjoint union of $\g$-strings, we can sum up all the scalar products with $\g$ and we get that $(\zeta_j,H_\g)=0$. Since $\sqrt{-1}\zeta_j$ belongs to the Cartan subalgebra of $\gh$ and is orthogonal to every $H_\g$, $\g\in R_\gh$, it lies in the center of $\gh$.\end{proof}
Now it is clear that for every $j=1,\ldots,s$ and for every $\a,\b\in R_j$ we have $g_\a=g_\b$ and this common value will be called $g_j$. We can 
write $\d_h = \sum_{j=1}^s\frac 1{g_j}\zeta_j$ and therefore it lies in  $\sqrt{-1}\gz$. Moreover, since $(\d_\kappa,H_\a)>0$ for every $\a\in R_\gn^+$ (see e.g. \cite{BFR}), we see that $(\d_\kappa,\d_h)>0$ and therefore $\d_h\neq 0$.
Our result is the following 
\begin{theorem} Let $G$ be a compact connected semisimple Lie group. 
\begin{itemize}
\item[i)] Let $M=G/L$ be a compact simply connected complex homogeneous space. Then an adapted $G$-invariant Hermitian metric $h$ on $M$ is balanced if and only if  $\sqrt{-1}\d_h$ lies in the center of $\gl$.
\item[ii)] Let $Q:=G/H$ be a flag manifold with $b_2(Q)\geq 3$. Then there exists a complex homogeneous space $G/L$ with Tits fibration $G/L\to Q$, which admits a balanced metric.
\end{itemize}\end{theorem}
In order to prove (ii), we recall some standard facts about flag manifolds and $T$-roots(see e.g. \cite{Al}). It is known that there exists a system of simple roots $\{\a_1,\ldots,\a_p,\b_1,\ldots,\b_t\}$ for $\gg$ such that $\{\a_1,\ldots,\a_k\}$ is a system of simple roots for $\gh_{ss}$ and $\b_1,\ldots,\b_t\in R_\gn^+$. Moreover we can reorder the modules $\gn_j$ ($j=1,\ldots,s$) so that $\b_j\in R_j$ for $j=1,\ldots,t$. This implies that 
$\{\zeta_1,\ldots,\zeta_t\}$ is a basis of $\sqrt{-1}\gz$ and there exist non negative integers $n_{ij}$ with 
$\zeta_i = \sum_{j=1}^t n_{ij}\zeta_j$ for $i=t+1,\ldots,s$. Now let $\Lambda$ be the integral lattice in $\gz$ given by the kernel of the exponential map. We can find 
$v\in \Lambda$ so that $\sqrt{-1}v = \sum_{j=1}^t c_j\zeta_j$ with $c_j >0$ for $j=1,\ldots,t$ and, up to a suitable scaling by a positive real number, we can suppose that $c_j>\sum_{k=t+1}^sn_{kj}$ for $j=1,\ldots,s$. We now put $g_i = 1$ for $i=t+1,\ldots,s$ and $\frac1{g_j}= c_j - \sum_{k=t+1}^sn_{kj} >0$, defining an invariant metric $h$ on $G/H$. The corresponding $\sqrt{-1}\d_h$ will therefore generate a one-dimensional line in $\gz$ which integrates to a closed one-dimensional torus $T$ by construction. Since $b_2(Q) = \dim \gz\geq 3$, we can find a torus $\tilde T$ (of dimension $1$ or $2$) with $T\subseteq \tilde T\subset Z(H)$ so that the codimension of $\tilde T$ in $Z(H)$ is even and this gives the isotropy $L = \tilde T\cdot H_{ss}$.\par
As a final remark, we note that the case when $b_2(Q)=2$ has been already treated in full generality in Proposition \ref{bal0}.

\vfil\eject

\bigskip\bigskip


\bigskip\bigskip
\begin{thebibliography}{25}

\bibitem{Al} D.V. Alekseevsky, {\it Flag manifolds}, Yugoslav Geometrical Seminar,
Divcibare, (1996), 3--35
\bibitem{AP} D.V. Alekseevsky and A.M. Perelomov, {\it  Invariant K\"ahler-Einstein metrics on compact
homogeneous spaces\/}, Funct. Anal. Applic., {\bf 20} (1986), 171--182
\bibitem{Ak} D. Akhiezer, Lie Group Actions in Complex Analysis, Aspects in Math. vol E27 Vieweg 1995
\bibitem{B} A. Balas, {\it Compact Hermitian manifolds of constant holomorphic sectional curvature}, Math. Z. {\bf 189}¢ (1985), 193--210
\bibitem{BFR} M. Bordermann, M. Forger and H. R\"omer,
{\it Homogeneous K\"ahler Manifolds: paving the way towards new supersymmetric Sigma Models\/}, Comm. Math. Phys. {\bf 102} (1986), 605--647
\bibitem{FG} A. Fino, G. Grantcharov, {\it Properties of manifolds with skew-symmetric torsion and special holonomy}, Adv.Math. {\bf 189} (2004), 439--450
\bibitem{G} P. Gauduchon, {\it Hermitian connections and Dirac operators}, Bull. U.M.I. B (7) {\bf 11} (1997) n.2 suppl. fasc. 2, 257--288
\bibitem{G1} P. Gauduchon, {\it La topologie d'une surface hermitienne d'Einstein} C.R. Acad.Sc. Paris  t. {\bf 290} (1980), 509--512
\bibitem{G2} P. Gauduchon, {\it La $1$-forme de torsion d'une vari\'et\'e hermitienne compacte } Math. Ann. {\bf 267} (1984), 495--518
\bibitem{GGP} D. Grantcharov, G. Grantcharov, Y.S. Poon, {\it Calabi-Yau connections with torsion on toric bundles}, J. Differential Geom. {\bf 78} (2008), 13--32
\bibitem{He} S. Helgason, Differential Geometry, Lie groups, and Symmetric spaces, Academic Press, Inc (1978)
\bibitem{Ko} S. Kobayashi, Transformation Groups in Differential Geometry, Classics in Math. Band 70, Springer Verlag  (1995)
\bibitem{KN} S. Kobayashi, K. Nomizu, Foundations of
differential geometry. Interscience Tracts in
 Pure and Applied Mathematics, No. 15 Vol. II   John Wiley Sons, Inc., New York-London-Sydney 1969
 \bibitem{KW} S. Kobayashi, H. Wu, {\it On holomorphic sections of certain Hermitian vector bundles}, Math. Ann. {\bf 189} (1970), 1--4
\bibitem{LY} K. Liu, X. Yang, {\it Geometry of Hermitian manifolds}, Internat.J.Math. {\bf 23} (2012), 12500553 
\bibitem{ST} J. Streets, G. Tian, {\it Hermitian curvature flows}, J.Eur.Math.Soc. {\bf 13} (2011), 601--634
\bibitem{ST1} P. Sankaran, A.S. Thakur, {\it Complex structures on product of circle bundles over complex manifolds}, Ann. Ist. Fourier {\bf 63} (2013), 1331-1366
\bibitem{T} K. Tsukada, {\it Eigenvalues of the Laplacian on Calabi-Eckmann manifolds}, J. Math. Soc. Japan {\bf 33} (1981), 673--691
\bibitem{U} Y. Ustinovskiy, {\it Hermitian curvature flow on manifolds with non-negative Griffiths curvature}, arXiv:1604.04813v1[math.CV],(2016)
\bibitem{W} H.-C. Wang, {\it Closed manifolds with homogeneous complex structures}, Amer. J. Math. {\bf 76} (1954), 1--32



\end{thebibliography}
\end{document}